\documentclass[10pt,fleqn,draft]{article}
\usepackage[dvips]{epsfig}
\input{epsf}
\usepackage{amsmath,amssymb,amsfonts}
\usepackage{amscd}
\usepackage[dvips,matrix,arrow]{xy}

\numberwithin{equation}{subsection}
\numberwithin{figure}{subsection}

\newtheorem{theorem}{Theorem}[subsection]
\newtheorem{lemma}[theorem]{Lemma}
\newtheorem{remark}[theorem]{Remark}
\newtheorem{proposition}[theorem]{Proposition}
\newtheorem{corollary}[theorem]{Corollary}

\def\pr{\operatorname{pr}}%
\def\dist{\operatorname{dist}}%
\def\eff{\operatorname{Eff}}%
\def\ef{\operatorname{Ef}}%
\def\min{\operatorname{min}}%

\newenvironment{proof}[1][Proof:]{\begin{trivlist}
\item[\hskip \labelsep {\bfseries #1}]}{\end{trivlist}}

\newenvironment{keyword}[1][Key-words:]{\begin{trivlist}
\item[\hskip \labelsep {\bfseries #1}]}{\end{trivlist}}

\newenvironment{codesams}[3][AMS codes:]{\begin{trivlist}
\item[\hskip \labelsep {\bfseries #1}]}{\end{trivlist}}

\newenvironment{codesjel}[3][JEL code:]{\begin{trivlist}
\item[\hskip \labelsep {\bfseries #1}]}{\end{trivlist}}

\newcommand{\qed}{\nobreak \ifvmode \relax \else
      \ifdim\lastskip<1.5em \hskip-\lastskip
      \hskip1.5em plus0em minus0.5em \fi \nobreak
      \vrule height0.75em width0.5em depth0.25em\fi}

 0 3 \font\ccc =msbm10 

\date{ }

\begin{document}

\title{On Markowitz Geometry}

\author{Valentin Vankov Iliev\\
Institute of Mathematics and Informatics,\\
Bulgarian Academy of Sciences, Sofia 1113, Bulgaria.\\
E-mail address: viliev@math.bas.bg}

\maketitle

\begin{abstract}

By Markowitz geometry we mean the intersection theory of ellipsoids
and affine subspaces in a real finite-dimensional linear space. In
the paper we give a meticulous and self-contained treatment of this
arch-classical subject, which lays a solid mathematical groundwork
of Markowitz mean-variance theory of efficient portfolios in
economics.

\end{abstract}


\begin{codesams}

\hbox{15A63; 15A99; 49K35}

\end{codesams}

\begin{codesjel}

    \hbox{G110}

\end{codesjel}

\section{Introduction and Notation}

\label{0}

\subsection{Introduction}

\label{0.1}

In this paper we solve the following extremal problem: Given a
positive dimensional affine subspace $C\subset\hbox{\ccc R}^n$, a
linear form $\pi$ which is not constant on $C$, and a positive
definite quadratic form $v$ on $\hbox{\ccc R}^n$, find all points
$x_0\in C$ such that
\begin{equation}
\pi(x_0)=\max_{x\in C, v\left(x\right)\leq v\left(x_0\right)}\pi(x)
\hbox{\rm\ and\ } v(x_0)=\mathop{\min}_{x\in C,
\pi\left(x\right)\geq \pi\left(x_0\right)}v(x).\label{0.1.1}
\end{equation}
It turns out that the locus of solutions of~(\ref{0.1.1}) is a ray
$E$ in $C$ whose endpoint $x_0$ is the foot of the perpendicular
$\varrho_0$ from the origin $O$ of the coordinate system to the
affine space $C$ (perpendicularity is with respect to the scalar
product obtained from $v$ via polarization). Let $h_r$ be the
hyperplane with equation $\pi(x)=r$, $r\in\hbox{\ccc R}$, and let
$\varrho_r$ be the perpendicular from $O$ to the affine subspace
$C\cap h_r$. If $\pi(x_0)=r_0$, $v(x_0)=a_0$, then
$E=\{\varrho_r\mid r\geq r_0\}$, $\varrho_0=\varrho_{r_0}$, and the
levels $r=\pi(x)$ and $a=v(x)$ are quadratically related along $E$:
$a=cr^2$.

Let $x={}^t(x_1,\ldots,x_n)$ be the generic vector in $\hbox{\ccc
R}^n$, let $M$ be a proper subset of the set $[n]=\{1,\ldots,n\}$ of
indices, and let $C=\cap_{j\in M}h^{\left(j\right)}$, where
$h^{\left(j\right)}$ are linearly independent hyperplanes with
equations
\begin{equation}
\pi^{\left(j\right)}(x)=\tau_j$,\hbox{\ } $\tau_j\in\hbox{\ccc R},
\hbox{\ }j\in M.\label{0.1.5}
\end{equation}
In case the hyperplane $\Pi=\{x\mid x_1+\cdots+x_n=1\}$ is one of
$h^{\left(j\right)}$'s, we may interpret $x\in C$ as an $n$-assets
financial portfolio, subject to the linear
constraints~(\ref{0.1.5}). Next, under certain conditions,
see~\ref{3.5}, we may interpret $\pi(x)$ as the expected return on
the portfolio $x$ and $v(x)$ as its risk. Finally, we may interpret
the elements of $E$ as efficient portfolios from Markowitz
mean-variance theory in economics, considered from purely
geometrical point of view. The famous pioneering work~\cite{[20]} is
written in this fashion and the condition for nonnegativity of the
variables (due to lack of short sales) distorts the picture there
and forces the use of variants of simplex method in Markowitz's
monograph~\cite{[25]}. Thus, instead of the ray $E$ of efficient
portfolios, we have to examine a more sophisticated piecewise set
$E_M$ of linear segments enclosed in the compact trace $\Delta$ of
the unit simplex in $\Pi$ on $C$. If $x_0\in E_M\backslash
E\cap\Delta$, then
\[
\pi(x_0)<\max_{x\in C, v\left(x\right)\leq v\left(x_0\right)}\pi(x)
\hbox{\rm\ or\ } v(x_0)>\mathop{\min}_{x\in C, \pi\left(x\right)\geq
\pi\left(x_0\right)}v(x),
\]
that is, the maximum $\pi(x_0)$ of the expected return decreases or
the minimum $v(x_0)$ of the risk increases, which is our point of
departure.

In section 1, Theorem~\ref{1.5.20}, we show that the trace $Q_a\cap
C$ of an ellipsoid $Q_a$ with equation $v(x)=a$ in $\hbox{\ccc R}^n$
on the affine space $C$ is again an ellipsoid in case $a\geq
\gamma_M(\tau)$, where $\gamma_M(\tau)$ is a positive definite
quadratic form in the variables $\tau=(\tau_j)_{j\in M}\in\hbox{\ccc
R}^M$. The center of the ellipsoid $Q_a\cap C$ is the foot of the
perpendicular $\varrho_0$ from $O$ to $C$, and, moreover, we find
its equation in terms of appropriate coordinates on $C$.

The inequality $a\geq\gamma_M(\tau)$ determines an "elliptic" cone
$\hat{\gamma}_M$ in $\hbox{\ccc R}\times\hbox{\ccc R}^M$, which is
the base of the bundle $\xi$ described in Theorem~\ref{1.10.10}. By
dragging the ellipsoids $a=\gamma_M(\tau)$ "upward" ($a$ is
increasing) we establish a real algebraic variety $\Gamma_M$ which
is the frontier of $\hat{\gamma}_M$ and branch locus of $\xi$. The
fibres of $\xi$ over the points in the interior of $\hat{\gamma}_M$
are ellipsoids which degenerate into their centers over $\Gamma_M$.
Using this bundle, we obtain that the image (the shadow) of an
ellipsoid in $\hbox{\ccc R}^n$ via projection parallel to some
subspace, is again an ellipsoid --- see Proposition~\ref{1.10.25}.

In section 2 we prove some extremal properties of the tangential
points of members of a family of eccentric ellipsoids and parallel
hyperplanes in $\hbox{\ccc R}^n$. These two sections stick together
in section 3 where we prove that the ray $E$ is the locus of all
efficient Markowitz portfolios and give interpretation of the
geometrical results in terms of Markowitz mean-variance theory.

\subsection{Notation}

\label{0.5}

For any positive integer $n$ we identify the members of the real
linear space $\hbox{\ccc R}^n$ with matrices of type $n\times 1$:
$x={}^t(x_1,\ldots,x_n)$, where the sign ${}^t$ means the transpose
of a matrix. We set $O={}^t(0,\ldots,0)\in \hbox{\ccc R}^n$ and
denote by $(e_i)_{i=1}^n$ the standard basis in $\hbox{\ccc R}^n$.
Say that $M=\{j_1,\ldots,j_m\}$, $j_1<\cdots<j_m$, be a proper
subset of the set of indices $[n]=\{1,\ldots,n\}$. Given a vector
$x={}^t(x_1,\ldots,x_n)$, we denote by $x^{\left(M\right)}$ the
vector ${}^t(x_{j_1},\ldots,x_{j_m})\in \hbox{\ccc R}^M$. Moreover,
indexed Greek letters $\tau^{\left(M\right)}$, etc., mean vectors
${}^t(\tau_{j_1},\ldots,\tau_{j_m})$, etc., from the linear space
$\hbox{\ccc R}^M$. In case $K$ is a proper subset of the set $M$ and
we fix all $\tau_j$, $j\in K$, and vary $\tau_j$, $j\in L$, where
$L=M\setminus K$, then, with some abuse of notation (the fixed
components are supposed to be known), we write
$\tau^{\left(M\right)}=\tau^{\left(L,K\right)}$.

Given a symmetric $n\times n$ matrix $Q$, by $Q^{\left(M\right)}$ we
denote the principal $m\times m$ submatrix of $Q$, obtained by
suppressing the rows and columns with indices which are not in $M$.

For a positive definite quadratic form $v(x)={}^t\!xQx$ on
$\hbox{\ccc R}^n$ with matrix $Q$ we denote $Q_a=\{ x\in \hbox{\ccc
R}^n\mid v(x)=a\}$, $a\geq 0$. The set $Q_a$ is an ellipsoid with
center $O$ in $\hbox{\ccc R}^n$ for all $a>0$. In case $n=1$ the
"ellipsoid" $Q_a$ consists of two (possibly coinciding) points. We
extend this terminology by defining the singleton $\{O\}$ to be an
"ellipsoid" when $a=0$ as well as in the case of zero-dimensional
linear space.

For any $a\geq 0$ we denote $Q_{\leq a}=\{x\in \hbox{\ccc R}^n\mid
v(x)\leq a\}$ and $Q_{<a}=\{x\in \hbox{\ccc R}^n\mid v(x)<a\}$. Note
that $Q_{\leq a}$ and $Q_{< a}$ are strictly convex sets.

We let $\pi(x)=p_1x_1+\cdots+p_nx_n$ be a linear form and let us
denote by $h_r$ the hyperplane in $\hbox{\ccc R}^n$, defined by the
equation $\pi(x)=r$, $r\in\hbox{\ccc R}$. Let $h_r(\leq)$ denote the
half-space $\{x\in \hbox{\ccc R}^n\mid \pi(x)\leq r\}$. The meaning
of notation $h_r(\geq)$, $h_r(<)$, and $h_r(>)$ is clear.

The standard scalar product $(x,y)={}^t\!xy$ in $\hbox{\ccc R}^n$
produces the standard norm $\|x\|$ with $\|x\|^2=(x,x)$. We set
$S^{n-1}=\{x\in \hbox{\ccc R}^n\mid \|x\|=1\}$ (the unit sphere).

The scalar product  $\langle x,y\rangle={}^t\!xQy$ in $\hbox{\ccc
R}^n$ produces the $Q$-norm $\|x\|_Q$ with $\|x\|_Q^2=\langle
x,x\rangle=v(x)$ and the $Q$-distance $\dist_Q(x,y)=\|x-y\|_Q$.
Thus, the ellipsoid $Q_a$ is a $Q$-sphere with $Q$-radius
$\sqrt{a}$. Two vectors $x$ and $y$ are said to be
$Q$-perpendicular, if $\langle x,y\rangle=0$.

Throughout the rest of the paper we assume that $n$ is a positive
integer and $m$ is a nonnegative integer with $m<n$. Moreover, we
suppose that if a proper subset $M$ of the set $[n]$ of indices is
given as a list: $M=\{j_1,\ldots,j_m\}$, then $j_1<\cdots<j_m$.

\section{Ellipsoids and Affine Subspaces}

\label{1}

\subsection{Intersections of Quadric Hypersurfaces\\
and Affine Subspaces }

\label{1.1}

Let $M\subset [n]$ be a set of indices of size $m$,
$M=\{j_1,\ldots,j_m\}$, and let $(h^{\left(j\right)})_{j\in M}$ be a
family of linearly independent affine hyperplanes in $\hbox{\ccc
R}^n$. The system of coordinates can be chosen in such a way that
the hyperplane $h^{\left(j\right)}$ has equation $x_j=\tau_j$,
$\tau_j\in\hbox{\ccc R}$. We denote by $h(\tau^{\left(M\right)})$
the intersection $\cap_{j\in M}h^{\left(j\right)}$. The family
$\{h(\tau^{\left(M\right)})\mid\tau^{\left(M\right)}\in\hbox{\ccc
R}^M\}$ consists of all $(n-m)$-dimensional affine spaces in
$\hbox{\ccc R}^n$, which are orthogonal to the $m$-dimensional
vector subspace generated by the vectors $e_j$, $j\in M$.

Let $Q=(q_{ij})_{i,j=1}^n$ be a symmetric matrix. For any $j\in M$
we denote by $\rho_{-,j}^{\left(Q;M^c\right)}$ the $j$-th column of
the $(n-m)\times n$ matrix obtained from $Q$ by deleting the rows
indexed by the elements of $M$. Thus,
$\rho_{-,j}^{\left(Q;M^c\right)}$ is a vector in $\hbox{\ccc
R}^{n-m}$ with components $\rho_{i,j}^{\left(Q;M^c\right)}=q_{ij}$,
$i\in M^c$. Given a vector $\tau^{\left(M\right)}\in\hbox{\ccc
R}^M$, $\tau^{\left(M\right)}={}^t(\tau_{j_1},\ldots,\tau_{j_m})$,
we set $\rho_{-,\tau^{\left(M\right)}}^{\left(Q;M^c\right)}=
\sum_{k=1}^m\tau_{j_k}\rho_{-,j_k}^{\left(Q;M^c\right)}$. By
\[
\alpha_{\left(Q;M\right)}(x)= \sum_{j,k\in M}^nq_{jk}x_jx_k
\]
we denote the quadratic form which corresponds to the principal
submatrix $Q^{\left(M\right)}$ of $Q$.

Let $M^c=\{i_1,\ldots,i_{n-m}\}$. In case the submatrix
$Q^{\left(M^c\right)}$ is invertible, let
\[
x^{\left(M^c\right)}= {}^t(c_{i_1}^{\left(Q;M^c\right)}
(\tau^{\left(M\right)}),\ldots,
c_{i_{n-m}}^{\left(Q;M^c\right)}(\tau^{\left(M\right)}))
\]
be the solution of the matrix equation
\begin{equation}
Q^{\left(M^c\right)}x^{\left(M^c\right)}=
-\rho_{-,\tau^{\left(M\right)}}^{\left(Q;M^c\right)}.\label{1.1.25}
\end{equation}
We set
\[
c^{\left(Q;M^c\right)}(\tau^{\left(M\right)})=
{}^t(c_1^{\left(Q;M^c\right)}(\tau^{\left(M\right)}),\ldots,
c_n^{\left(Q;M^c\right)}(\tau^{\left(M\right)})),
\]
where $c_j^{\left(Q;M^c\right)}(\tau^{\left(M\right)})=\tau_j$ for
$j\in M$. In particular,
$c^{\left(Q;M^c\right)}(\tau^{\left(M\right)})\in
h(\tau^{\left(M\right)})$. In case $L\subset M$,
$L=\{\ell_1,\ldots,\ell_{\lambda}\}$, we set
\[
c_L^{\left(Q;M^c\right)}(\tau^{\left(M\right)})=
{}^t(c_{\ell_1}^{\left(Q;M^c\right)}(\tau^{\left(M\right)}),\ldots,
c_{\ell_{\lambda}}^{\left(Q;M^c\right)}(\tau^{\left(M\right)})).
\]

Note that if $M=\emptyset$, then
$c^{\left(Q;[n]\right)}(\tau^{\left(\emptyset\right)})=0$. We write
$c^{\left(Q;M^c\right)}(\tau^{\left(M\right)})=
c^{\left(M^c\right)}(\tau)$, and, similarly,
$\rho_{-,\tau^{\left(M\right)}}^{\left(Q;M^c\right)}=
\rho_{-,\tau}^{\left(M^c\right)}$, etc., when the context allows
that.

Since the vector $\rho_{-,\tau}^{\left(M^c\right)}\in\hbox{\ccc
R}^{n-m}$ depends linearly on $\tau^{\left(M\right)}$, the map
\begin{equation}
\psi_M\colon\hbox{\ccc R}^M\to \hbox{\ccc R}^n,\hbox{\
}\tau^{\left(M\right)}\mapsto
c^{\left(M^c\right)}(\tau^{\left(M\right)}),\label{1.1.35}
\end{equation}
is an injective homomorphism of linear spaces. We set
$\eff^{\left(Q;M^c\right)}=\psi_M(\hbox{\ccc R}^M)$ and note that
$\eff^{\left(Q;M^c\right)}$ is an $m$-dimensional subspace of
$\hbox{\ccc R}^n$. Below we use also the short notation
$\eff^{\left(M^c\right)}=\eff^{\left(Q;M^c\right)}$ when the matrix
$Q$ is given by default.

\begin{lemma}\label{1.1.40} Let $K$ and $M$ be proper subsets
of the set of indices $[n]$ with $K\subset M$. Let
$Q^{\left(K^c\right)}$ and $Q^{\left(M^c\right)}$ be invertible
submatrices of $Q$. The following two statements are equivalent:

{\rm (i)} One has $c^{\left(K^c\right)}(\tau)\in
h(\tau^{\left(M\right)})$.

{\rm (ii)} One has $c^{\left(K^c\right)}(\tau)=
c^{\left(M^c\right)}(\tau)$.

\end{lemma}

\begin{proof} We have $h(\tau^{\left(M\right)})\subset
h(\tau^{\left(K\right)})$ and let us assume $K\neq M$. It is enough
to prove that ${\rm (i)}$ implies ${\rm (ii)}$. Let
$c^{\left(K^c\right)}(\tau)\in h(\tau^{\left(M\right)})$. We remind
that the hyperplane $h^{\left(j\right)}$ has equation
$h^{\left(j\right)}\colon x_j=\tau_j$ for any $j\in M$. In
particular, for each $j\in K^c\setminus M^c=M\setminus K$ we obtain
$c_j^{\left(K^c\right)}(\tau)=\tau_j$. Therefore
$c^{\left(K^c\right)}(\tau)_{M^c}$ is a solution of the
equation~(\ref{1.1.25}). The uniqueness of this solution implies
$c^{\left(K^c\right)}(\tau)= c^{\left(M^c\right)}(\tau)$.

\end{proof}

\begin{corollary}\label{1.1.45}  One has
\[
\eff^{\left(K^c\right)}\cap
h(\tau^{\left(M\right)})\subset\eff^{\left(M^c\right)}.
\]

\end{corollary}

Now, let us fix all components of
$\tau^{\left(M\right)}\in\hbox{\ccc R}^M$, except $r=\tau_\ell$ for
some $\ell\in M$, so $\tau^{\left(M\right)}=
\tau^{\left(\{\ell\},M\setminus\{\ell\}\right)}(r)$. When we vary
$r\in\hbox{\ccc R}$, then
$\tau^{\left(\{\ell\},M\setminus\{\ell\}\right)}(r)$ describes a
straight line in $\hbox{\ccc R}^M$ and hence
$c^{\left(M^c\right)}(\tau^{\left(\{\ell\},M\setminus\{\ell\}\right)}(r))$
describes a straight line in $\hbox{\ccc R}^n$ which we denote by
$\eff_\ell^{\left(Q;M^c\right)}$. Its ray
$\{c^{\left(M^c\right)}(\tau^{\left(\{\ell\},M\setminus\{\ell\}\right)}(r))
\mid r\geq b\}$, $b\in\hbox{\ccc R}$, is denoted by
$\eff_{\ell^{b+}}^{\left(Q;M^c\right)}$.

Let us set
\[
\gamma_M^{\left(Q\right)}(\tau)= \alpha_{\left(Q;M\right)}(\tau)-
\alpha_{\left(Q;M^c\right)}(c_{i_1}^{\left(Q;M^c\right)}
(\tau),\ldots, c_{i_{n-m}}^{\left(Q;M^c\right)}(\tau)).
\]
Since $\alpha_{\left(Q;\emptyset\right)}(x)=0$ and
$c_1^{\left(Q;[n]\right)}(\tau)=\cdots=
c_n^{\left(Q;[n]\right)}(\tau)=0$, we obtain
$\gamma_\emptyset^{\left(Q\right)}(\tau)=0$. We write
$\gamma_M^{\left(Q\right)}(\tau)=\gamma_M(\tau)$ when the matrix $Q$
is known from the context.

It follows from Lemma~\ref{A.1.5}, {\rm (i)}, that $\gamma_M(\tau)$
is a quadratic form in $\tau^{\left(M\right)}$.

Let us move the origin of the coordinate system by the substitution
$x=z(\tau^{\left(M\right)})+c^{\left(M^c\right)}(\tau^{\left(M\right)})$.
Then the restrictions of the components of both
$x^{\left(M^c\right)}$ and
$z^{\left(M^c\right)}(\tau^{\left(M\right)})$ on
$h(\tau^{\left(M\right)})$ are coordinate functions in this
$(n-m)$-dimensional affine space.

Let $v(x)={}^t\!xQx$ be the quadratic form produced by the symmetric
nonzero $n\times n$-matrix $Q$. Thus, $Q_a\colon v(x)=a$ is a
quadric in $\hbox{\ccc R}^n$ for generic $a\in\hbox{\ccc R}$ and the
real variety $q_{a,\tau^{\left(M\right)}}=Q_a\cap
h(\tau^{\left(M\right)})$ is defined in $h(\tau^{\left(M\right)})$
by the equation
\begin{equation}
{}^tx^{\left(M^c\right)}Q^{\left(M^c\right)}x^{\left(M^c\right)}+
2{}^t\!\rho_{-,\tau}^{\left(M^c\right)}
x^{\left(M^c\right)}+\alpha_M(\tau)-a=0. \label{1.1.60}
\end{equation}

Let us set \begin{equation}
v^{\left(M^c\right)}(z(\tau^{\left(M\right)}))=
{}^tz^{\left(M^c\right)}(\tau^{\left(M\right)})
Q^{\left(M^c\right)}z^{\left(M^c\right)}(\tau^{\left(M\right)}).
\label{1.1.65}
\end{equation}
In case the principal submatrix $Q^{\left(M^c\right)}$ is
invertible, Lemma~\ref{A.1.10} implies that
$v(x)=v^{\left(M^c\right)}(z(\tau^{\left(M\right)}))+
\gamma_M(\tau^{\left(M\right)})$ on $h^{\left(M\right)}$, and in
terms of $z$-coordinates the equation~(\ref{1.1.60}) has the form
\begin{equation}
v^{\left(M^c\right)}(z(\tau^{\left(M\right)}))=
a-\gamma_M(\tau^{\left(M\right)}).\label{1.1.70}
\end{equation}

\subsection{Intersections of Ellipsoids and Affine Subspaces}

\label{1.5}

Let $v(x)={}^t\!xQx$ be a positive definite quadratic form produced
by the symmetric (positive definite) $n\times n$-matrix $Q$. This
being so, $Q_a\colon v(x)=a$ is an ellipsoid in $ \hbox{\ccc R}^n$
for $a>0$, $Q_0=\{0\}$, and $Q_a=\emptyset$ for $a<0$. In
particular, $Q^{\left(M^c\right)}$ is a principal, hence positive
definite, submatrice of $Q$. Thus, the quadratic form~(\ref{1.1.65})
is positive definite.

In accord with~(\ref{1.1.60}) and~(\ref{1.1.70}), we establish parts
{\rm (ii)}, {\rm (iii)}, and {\rm (iv)} of the next theorem. Part
{\rm (i)} is proved in Lemma~\ref{A.1.5}, {\rm (ii)}.

\begin{theorem}\label{1.5.20} Let the quadratic form $v(x)={}^t\!xQx$
be positive definite.

{\rm (i)} If $M\neq\emptyset$, then the quadratic form
$\gamma_M(\tau)$ is positive definite.

{\rm (ii)} If $a>\gamma_M(\tau)$, then $q_{a,\tau^{\left(M\right)}}$
is an ellipsoid in the $(n-m)$-dimensional vector space
$h(\tau^{\left(M\right)})$ with center $c^{\left(M^c\right)}(\tau)$
and $Q^{\left(M^c\right)}$-radius $\sqrt{a-\gamma_M(\tau)}$.

{\rm (iii)} If $a=\gamma_M(\tau)$, then
$q_{a,\tau^{\left(M\right)}}=\{c^{\left(M^c\right)}(\tau)\}$.

{\rm (iv)} If $a<\gamma_M(\tau)$, then the set
$q_{a,\tau^{\left(M\right)}}$ is empty.

\end{theorem}

\begin{remark}\label{1.5.25} {\rm We remind that ellipsoid in an
one-dimensional affine subspace is a set consisting of two points
and its center is the midpoint.

}

\end{remark}

\begin{remark}\label{1.5.40} {\rm  In accord with Lemma~\ref{2.1.10},
the affine subspace $h(\tau^{\left(M\right)})$ is tangential to the
ellipsoid $Q_a$, $a=\gamma_M(\tau)$, at the point
$x=c^{\left(M^c\right)}(\tau)$.

}

\end{remark}

\begin{remark}\label{1.5.45} {\rm  In view of the previous remark,
Lemma~\ref{1.1.40} has transparent geometrical meaning: If the
subspace $h(\tau^{\left(M\right)})$ of $h(\tau^{\left(K\right)})$
passes through the point $x=c^{\left(K^c\right)}(\tau)$, then
$h(\tau^{\left(M\right)})$ is also tangential to $Q_a$ at $x$.

}

\end{remark}

We obtain immediately the following corollary:

\begin{corollary}\label{1.5.50} {\rm (i)} For any
$x\in h(\tau^{\left(M\right)})$ one has $v(x)\geq \gamma_M(\tau)$
and an equality holds if and only if $x=c^{\left(M^c\right)}(\tau)$.

{\rm (ii)} The point $c^{\left(M^c\right)}(\tau)\in
h(\tau^{\left(M\right)})$ is the foot of $Q$-perpendicular from the
origin $O$ to the affine subspace $h(\tau^{\left(M\right)})$ and one
has
\[
\dist_Q(O,h(\tau^{\left(M\right)}))=
\|c^{\left(M^c\right)}(\tau)\|_Q= \sqrt{\gamma_M(\tau)}.
\]
\end{corollary}

\begin{corollary}\label{1.5.55} Let $K$ and $L$ be disjoint subsets
of $M$ with $K\cup L=M$. One has

{\rm (i)} If $a=\gamma_M(\tau^{\left(M\right)})$, then the trace
$q_{a,\tau^{\left(K\right)}}$ of the ellipsoid $Q_a$ on the affine
space $h(\tau^{\left(K\right)})$ is nonempty and the affine subspace
$h(\tau^{\left(M\right)})\subset h(\tau^{\left(K\right)})$ is
tangential to the ellipsoid $q_{a,\tau^{\left(K\right)}}$ at the
point $c^{\left(Q;M^c\right)}(\tau^{\left(M\right)})$.
\[
{\rm (ii)}\hbox{\ }
c^{\left(Q;M^c\right)}(\tau^{\left(M\right)})=c^{\left(Q;K^c\right)}
(\tau^{\left(K\right)})+
c^{\left(Q^{\left(K^c\right)};M^c\right)}(\tau^{\left(L\right)}-
c_L^{\left(Q;K^c\right)}(\tau^{\left(K\right)}))
\]
and
\[
{\rm (iii)}\hbox{\
}\gamma_M^{\left(Q\right)}(\tau^{\left(M\right)})=
\gamma_K^{\left(Q\right)}(\tau^{\left(K\right)})+
\gamma_L^{\left(Q^{\left(K^c\right)}\right)}(\tau^{\left(L\right)}-
c_L^{\left(Q;K^c\right)}(\tau^{\left(K\right)})).
\]

\end{corollary}

\begin{proof} Both assertions hold when one of the sets $M$,
$K$, or $L$, is empty.

{\rm (i)} The equalities
\[
q_{a,\tau^{\left(M\right)}}=q_{a,\tau^{\left(K\right)}}\cap
h(\tau^{\left(L\right)})=q_{a,\tau^{\left(K\right)}}\cap
h(\tau^{\left(M\right)})=Q_a\cap h(\tau^{\left(M\right)})
\]
and Theorem~\ref{1.5.20}, {\rm (ii)} -- {\rm (iv)}, yield that under
the condition $a=\gamma_M(\tau^{\left(M\right)})$ we have
\begin{equation}
q_{a,\tau^{\left(K\right)}}\cap h(\tau^{\left(M\right)})=
\{c^{\left(Q;M^c\right)}(\tau^{\left(M\right)})\}.\label{1.5.60}
\end{equation}

In particular, $a\geq\gamma_K(\tau^{\left(K\right)})$ and in this
case $q_{a,\tau^{\left(K\right)}}$ is an ellipsoid in the vector
space $h(\tau^{\left(K\right)})$ endowed with coordinate functions
$(z_s^{\left(K^c\right)}(\tau^{\left(K\right)}))_{s\in K^c}$. The
point $\{c^{\left(Q;K^c\right)}(\tau^{\left(K\right)})\}$ is both
the origin of the coordinates and the center of the ellipsoid
$q_{a,\tau^{\left(K\right)}}$ which has equation
\[
{}^tz^{\left(K^c\right)}(\tau^{\left(K\right)})
Q^{\left(K^c\right)}z^{\left(K^c\right)}(\tau^{\left(K\right)})=
a-\gamma_K(\tau^{\left(K\right)}).
\]
Therefore we have
\[
q_{a,\tau^{\left(K\right)}}=
Q_{a-\gamma_K(\tau^{\left(K\right)})}^{\left(K^c\right)}.
\]
Because of~(\ref{1.5.60}), the trace $h(\tau^{\left(M\right)})$ of
$h(\tau^{\left(L\right)})$ on $h(\tau^{\left(K\right)})$ is
tangential to $q_{a,\tau^{\left(K\right)}}$ at the point
$c^{\left(Q;M^c\right)}(\tau^{\left(M\right)})$ (Note that in case
$q_{a,\tau^{\left(K\right)}}=
\{c^{\left(Q;M^c\right)}(\tau^{\left(M\right)})\}$ we have
$c^{\left(Q;M^c\right)}(\tau^{\left(M\right)})=
c^{\left(Q;K^c\right)}(\tau^{\left(K\right)})$ and
$h(\tau^{\left(M\right)})$ is also tangential to
$q_{a,\tau^{\left(K\right)}}$ at the point
$c^{\left(Q;M^c\right)}(\tau^{\left(M\right)})$ --- see
Remark~\ref{2.1.5}).

{\rm (ii)} The affine subspace $h(\tau^{\left(M\right)})$ is defined
in $h(\tau^{\left(K\right)})$ by the equations
$z_s^{\left(K^c\right)}(\tau^{\left(K\right)})=
\tau_s-c_s^{\left(Q;K^c\right)}(\tau^{\left(K\right)})$, $s\in L$
(we have $L\subset K^c$). Hence the difference
$c^{\left(Q;M^c\right)}(\tau^{\left(M\right)})-
c^{\left(Q;K^c\right)}(\tau^{\left(K\right)})$ of points in the
affine subspace $h(\tau^{\left(K\right)})\subset\hbox{\ccc R}^n$
coincides with the vector
$c^{\left(Q^{\left(K^c\right)};M^c\right)}(\tau^{\left(L\right)}-
c_L^{\left(Q;K^c\right)}(\tau^{\left(K\right)}))$ and we have
obtained part {\rm (ii)}. The equalities
$a-\gamma_K(\tau^{\left(K\right)})=
\gamma_L^{\left(Q^{\left(K^c\right)}\right)}(\tau^{\left(L\right)}-
c_L^{\left(Q;K^c\right)}(\tau^{\left(K\right)}))$ and
$a=\gamma_M(\tau^{\left(M\right)}$ yield assertion {\rm (iii)}.

\end{proof}

\begin{remark}\label{1.5.65} {\rm Since the vector
$c^{\left(Q^{\left(K^c\right)}\right)}(\tau^{\left(K\right)})$ is
$Q$-perpendicular to the affine subspace $h(\tau^{\left(K\right)})$
and since the vector
$c^{\left(Q^{\left(K^c\right)};M^c\right)}(\tau^{\left(L\right)}-
c_L^{\left(Q;K^c\right)}(\tau^{\left(K\right)}))$ lies in this
subspace, part {\rm (ii)} of the above corollary is Pythagorean
theorem.

}

\end{remark}

\begin{remark}\label{1.5.70} {\rm It follows from Theorem of three
perpendiculars that the vector
$c^{\left(Q^{\left(K^c\right)};M^c\right)}(\tau^{\left(L\right)}-
c_L^{\left(Q;K^c\right)}(\tau^{\left(K\right)}))$ is
$Q$-perpendicular to the affine subspace $h(\tau^{\left(M\right)})$.

}

\end{remark}

\subsection{ A Bundle}

\label{1.10}

Let us consider the $(m+1)$-dimensional space $\hbox{\ccc
R}\times\hbox{\ccc R}^M$ with generic vector
${}^t(a,\tau^{\left(M\right)})$, endowed with standard topology and
let $\hat{\gamma}_M=\{{}^t(a,\tau^{\left(M\right)})\in\hbox{\ccc
R}\times\hbox{\ccc R}^M\mid a\geq\gamma_M(\tau)\}$. The set
$\hat{\gamma}_M$ is the closed region in $\hbox{\ccc
R}\times\hbox{\ccc R}^M$, which consists of all points above the
graph $\Gamma_M$ of the quadratic function $a=\gamma_M(\tau)$ when
$M\neq\emptyset$ and $\hat{\gamma}_\emptyset=[0,\infty)\times\{0\}$.
In all cases $\pr_a(\hat{\gamma}_M)=[0,\infty)$. The set $\Gamma_M$
is an algebraic variety (hence a closed set) in $\hbox{\ccc
R}\times\hbox{\ccc R}^M$ and the difference
$\tilde{\gamma}_M=\hat{\gamma}_M\backslash\Gamma_M$ is an open set,
both being nonempty.

Let $\gamma_M(\tau)=
{}^t\tau^{\left(M\right)}R\tau^{\left(M\right)}$, where $R$ is a
symmetric $M\times M$-matrix. In accord with Theorem~\ref{1.5.20},
{\rm (i)}, in case $M\neq\emptyset$, the matrix $R$ is positive
definite. If $M=\emptyset$, then $R$ is the empty matrix. Given
$a\geq 0$, we set $R_a=\{\tau^{\left(M\right)}\in\hbox{\ccc R}^M\mid
\gamma_M(\tau^{\left(M\right)})=a\}$ and note that $R_a$ is an
ellipsoid in $\hbox{\ccc R}^M$. Any level set
$\Gamma_{a,M}=\{{}^t(a,\tau^{\left(M\right)})\in\hbox{\ccc
R}\times\hbox{\ccc R}^M\mid a=\gamma_M(\tau)\}$, $a>0$, is
isomorphic to the ellipsoid $R_a$ in $\hbox{\ccc R}^M$, and
$\Gamma_{0,M}=\{(0,0)\}$. Given $a\geq 0$, let us denote
$\eff^{\left(a;M^c\right)}=\{x\in \hbox{\ccc R}^n\mid x=
c^{\left(M^c\right)}(\tau),\hbox{\ }
{}^t(a,\tau^{\left(M\right)})\in\Gamma_{a,M}\}$. We define a
morphism of real algebraic varieties by the rule
\[
\varphi_M\colon \hbox{\ccc R}^n\to\hbox{\ccc R}\times\hbox{\ccc
R}^M,\hbox{\ } x\mapsto {}^t(v(x),x^{\left(M\right)}).
\]
Theorem~\ref{1.5.20} yields $\varphi_M(\hbox{\ccc
R}^n)=\hat{\gamma}_M$, we set
$\Phi_M=\varphi_M^{-1}(\hat{\gamma}_M)$, and denote the restriction
of $\varphi_M$ on $\Phi_M$ by the same letter. Since
$\varphi_M^{-1}({}^t(a,\tau^{\left(M\right)}))=
q_{a,\tau^{\left(M\right)}}$, we establish the following:

\begin{theorem}\label{1.10.10} Let $\xi=(\Phi_M,
\varphi_M,\hat{\gamma}_M)$ be the bundle defined by the map
$\varphi_M$.

{\rm (i)} The restriction $\xi_{\mid\tilde{\gamma}_M}$ is a
fibration with fibres
$\varphi_M^{-1}({}^t(a,\tau^{\left(M\right)}))=
q_{a,\tau^{\left(M\right)}}$,
${}^t(a,\tau^{\left(M\right)})\in\tilde{\gamma}_M$, which are
ellipsoids in $\hbox{\ccc R}^{n-m}$ with centers
$c^{\left(M^c\right)}(\tau)$.

{\rm (ii)} The restriction $\xi_{\mid\Gamma_M}$ is an isomorphism of
real algebraic $m$-dimensional varieties with inverse isomorphism
$\Gamma_M\to\eff^{\left(M^c\right)}$,
${}^t(a,\tau^{\left(M\right)})\mapsto c^{\left(M^c\right)}(\tau)$,
which maps any level set $\Gamma_{a,M}$ onto
$\eff^{\left(a;M^c\right)}$.

\end{theorem}

\begin{corollary} \label{1.10.15} The set $\eff^{\left(a;M^c\right)}$
is a real algebraic subvariety of $Q_a$, which is isomorphic via
$\xi\mid\Gamma_M$ to the ellipsoid $\Gamma_{a,M}$.

\end{corollary}

Taking into account Remark~\ref{1.5.40}, we obtain immediately the
following:

\begin{corollary} \label{1.10.20} The family
$\{h(\tau^{\left(M\right)})\mid\tau^{\left(M\right)}\in\Gamma_{a,M}\}$
consists of all $(n-m)$-dimensional affine spaces in $\hbox{\ccc
R}^n$, which are both orthogonal to the $m$-dimensional vector
subspace generated by the vectors $e_j$, $j\in M$, and tangential to
the ellipsoid $Q_a$.

\end{corollary}

\subsection{ A Shadow}

\label{1.15}

Let us denote by $\zeta_M$ the restriction of the second projection
$\pr_2\colon\hbox{\ccc R}\times\hbox{\ccc R}^M\to\hbox{\ccc R}^M$ on
$\hat{\gamma}_M$. The composition $\phi_M=\zeta_M\circ\varphi_M$ is
the restriction on $\Phi_M$ of the projection of $\hbox{\ccc R}^n$
parallel to the subspace $W$ defined by $x^{\left(M\right)}=0$:
$\phi_M\colon \hbox{\ccc R}^n\to W^\perp$,
$\phi_M(x)=x^{\left(M\right)}$, and, moreover,
$\phi_M^{-1}(\tau^{\left(M\right)})=h(\tau^{\left(M\right)})$. Since
the set $\eff^{\left(a;M^c\right)}\subset Q_a$ is mapped via
$\phi_M$ onto the ellipsoid $R_a$ in $\hbox{\ccc R}^M$ and since the
internal points of $Q_a$ are mapped onto the internal points of
$R_a$, we can formulate the result from Corollary~\ref{1.10.20} as
solution of a shadow problem:

\begin{proposition} \label{1.10.25} All $(n-m)$-dimensional
affine spaces in $\hbox{\ccc R}^n$ with common direction vector
subspace $W$, which are also tangential to an ellipsoid $Q_a$ in
$\hbox{\ccc R}^n$, intersect the orthogonal complement $W^\perp$ at
the points of an ellipsoid $R_a$ in $W^\perp\simeq\hbox{\ccc R}^M$.
All affine spaces in $\hbox{\ccc R}^n$ which have nonempty
intersection with the interior of $Q_a$ and are parallel to $W$
intersect $W^\perp$ at the internal points of $R_a$.

\end{proposition}

\section{Ellipsoids and Hyperplanes}

\label{2}

\subsection{Ellipsoids and their Tangent Spaces}

\label{2.1}

Let $v(x)={}^t\! xQx$ be a positive definite quadratic form. The
equation of the tangent space $\theta_{x_0}$ of the ellipsoid
$Q_a\colon v(x)=a$, $a>0$, at the point $x_0\in Q_a$ is
\[
\theta_{x_0}(x)=a,
\]
where $\theta_{x_0}(x)={}^t\!x_0Qx$. For all $x\in Q_a$ we have
$x\neq 0$ and since the matrix $Q$ has rank $n$, we obtain $Qx_0\neq
0$. In particular, $\theta_{x_0}$ is a hyperplane and $Q_a$ is a
smooth hypersurface in $\hbox{\ccc R}^n$.

\begin{remark}\label{2.1.5} The tangent space of the
"ellipsoid" $Q_0=\{O\}$ at its only point $x_0=O$ is $\hbox{\ccc
R}^n$. In particular, any linear subspace of $\hbox{\ccc R}^n$ is
tangential to $Q_0$.

\end{remark}

Let $a>0$ and let us fix a point $x_0\in Q_a$. For any vector $u\in
S^{n-1}$ we denote for short by $L_u$ the line $\{z\in \hbox{\ccc
R}^n\mid z=x_0+tu,\hbox{\ }t\in\hbox{\ccc R}\}$.

\begin{lemma}\label{2.1.10} One has
\[
L_u\cap Q_{\leq a}=\{x_0+tu\mid 0\leq t\leq
-2\frac{\theta_{x_0}(u)}{v(u)}\},\hbox{\ }L_u\cap Q_a=\{x_0,
x_0-2\frac{\theta_{x_0}(u)}{v(u)}u\}.
\]

\end{lemma}

\begin{proof} The inequality $v(x_0+tu)\leq a$ is equivalent to
$2\theta_{x_0}(u)t+v(u)t^2\leq 0$ and the equality holds if and only
if $t=0$ or $t=-2\frac{\theta_{x_0}(u)}{v(u)}$.

\end{proof}

\begin{lemma}\label{2.1.20} Let $x_0\in Q_a$.

{\rm (i)} One has $Q_{\leq a}\subset \theta_{x_0}(\leq)$.

{\rm (ii)} One has $Q_{\leq a}\cap \theta_{x_0}=Q_a\cap
\theta_{x_0}=\{x_0\}$.

{\rm (iii)} One has $Q_{\leq a}\backslash\{x_0\}\subset
\theta_{x_0}(<)$.

\end{lemma}

\begin{proof} {\rm (i)} Let $y\in Q_{\leq a}$, $y\neq x_0$, and let
$y\in L_u$. In accord with Lemma~\ref{2.1.10}, $y=x_0+tu$ where
$0\leq t\leq -2\frac{\theta_{x_0}(u)}{v(u)}$. We have
$\theta_{x_0}(y)=\theta_{x_0}(x_0)+
t\theta_{x_0}(u)=a+t\theta_{x_0}(u)\leq
a-2\frac{(\theta_{x_0}(u))^2}{v(u)}\leq a$.

{\rm (ii)} Let us suppose that there exists a point $y$, $y\neq
x_0$, with $y\in Q_{\leq a}\cap \theta_{x_0}$ and let
$u=\frac{1}{\|y-x_0\|}(y-x_0)$. Then $\theta_{x_0}(u)=0$, $y\in
L_u$, and Lemma~\ref{2.1.10} implies $L_u\cap Q_{\leq a}=\{x_0\}$
--- a contradiction with $y\in L_u\cap Q_{\leq a}$. Now, because of
the inclusions $\{x_0\}\subset Q_a\cap \theta_{x_0}\subset Q_{\leq
a}\cap \theta_{x_0}=\{x_0\}$, part {\rm (ii)} is proved.

Parts {\rm (i)} and {\rm (ii)} yield part {\rm (iii)}.

\end{proof}

We remind that $h_r$ is a hyperplane in $\hbox{\ccc R}^n$, defined
by the equation $\pi(x)=r$, where $\pi(x)$ is a non-zero linear
form, and $q_{a,r}=Q_a\cap h_r$.

\begin{lemma}\label{2.1.30} Let $x_0\in q_{a,r}$.

{\rm (i)} If $Q_a\subset h_r(\leq)$, then $h_r=\theta_{x_0}$.

{\rm (ii)} If $Q_{<a_0}\subset h_{r_0}(<)$, then $Q_a\subset
h_r(\leq)$.

\end{lemma}

\begin{proof} {\rm (i)} When $y$ varies through
$Q_a\backslash\{x_0\}$, then $u=\frac{1}{\|y-x_0\|}(y-x_0)$ varies
bijectively through $S^{n-1}\cap\theta_{x_0}(<)$. On the other hand,
since $Q_a\subset h_r(\leq)$, then $y\in Q_a\backslash\{x_0\}$
yields $\pi(y)\leq r$, that is,
$\pi(x_0-2\frac{\theta_{x_0}(u)}{v(u)}u)\leq r$, and hence
$\theta_{x_0}(u)\pi(u)\geq 0$ for all $u\in
S^{n-1}\cap\theta_{x_0}(<)$. The last inequality also holds for all
$u\in S^{n-1}\cap\theta_{x_0}(>)$ because
$\theta_{x_0}(-u)\pi(-u)\geq 0$. Thus, we have
$\theta_{x_0}(u)\pi(u)\geq 0$ for all $u\in S^{n-1}$, therefore for
all vectors $u\in \hbox{\ccc R}^n$. If the linear forms
$\theta_{x_0}$ and $\pi$ are not proportional, then after an
appropriate change of the coordinates, $\theta_{x_0}$ and $\pi$ can
serve as coordinate functions in $\hbox{\ccc R}^n$ --- a
contradiction.

{\rm (ii)} Let $y\in Q_a$ and let us set $y_n=(1-\frac{1}{n})y$ for
any positive integer $n$. Then $y_n\in Q_{<a_0}$ and
$\lim_{n\to\infty}y_n=y$. Since $Q_{<a_0}\subset h_{r_0}(<)$, we
obtain $h_{r_0}(y_n)<r_0$, hence $h_{r_0}(y)\leq r_0$.

\end{proof}

\subsection{Some Extremal Properties}

\label{2.2}

Let $h_r\colon\pi(x)=r$ be a hyperplane in $\hbox{\ccc R}^n$,
$\pi(x)=p_1x_1+\cdots+p_nx_n$, and let us set
$p={}^t(p_1,\ldots,p_n)$. We denote $q_{a,r}=Q_a\cap h_r$.

\begin{lemma}\label{2.2.1} Let $x_0\in \hbox{\ccc R}^n
\backslash\{0\}$, $a>0$, and $r>0$. The following four statements
are equivalent:

{\rm (i)} One has $x_0\in q_{a,r}$ and $Qx_0\in\hbox{\ccc R}p$.

{\rm (ii)} One has $rQx_0=ap$ and $a=r^2({}^t\!pQ^{-1}p)^{-1}$.

{\rm (iii)} One has $x_0\in q_{a,r}$ and $\theta_{x_0}=h_r$.

{\rm (iv)} One has $q_{a,r}=\{x_0\}$.

\end{lemma}

\begin{proof} ${\rm (i)}\Longrightarrow {\rm (ii)}$ Let
$Qx_0=bp$, $b\in\hbox{\ccc R}$. We have
\[
a=v(x_0)=^t\!x_0Qx_0= {}^t\!x_0(bp)=b{}^t\!px_0=b\pi(x_0)=br,
\]
therefore $rQx_0=ap$. On the other hand, we obtain
\[
a={}^t\!x_0Qx_0=\frac{a}{r}{}^t\!pQ^{-1}\frac{a}{r}p=
\frac{a^2}{r^2}{}^t\!pQ^{-1}p,
\]
hence $a=r^2({}^t\!pQ^{-1}p)^{-1}$.

${\rm (ii)}\Longrightarrow {\rm (i)}$ We have $Qx_0\in\hbox{\ccc
R}p$, and, moreover, ${}^t\!x_0=\frac{a}{r}{}^t\!pQ^{-1}$.
$\pi(x_0)={}^t\!px_0={}^t\!x_0p=\frac{a}{r}{}^t\!pQ^{-1}p=
\frac{a}{r}\frac{r^2}{a}=r$, hence $x_0\in h_r$. Finally,
$v(x_0)={}^t\!x_0Qx_0=\frac{a}{r}{}^t\!pQ^{-1}Qx_0=
\frac{a}{r}{}^t\!px=\frac{a}{r}\pi(x_0)=a$, therefore $x_0\in Q_a$.

The equivalence of parts {\rm (i)} and {\rm (iii)} is
straightforward. Part {\rm (iii)} and Lemma~\ref{2.1.20}, {\rm
(ii)}, imply part {\rm (iv)}.

${\rm (iv)}\Longrightarrow {\rm (iii)}$  Let $L=\{x_0+t z\mid
t\in\hbox{\ccc R}\}$, $z\neq 0$, be a line in $h_r$, that is,
$\pi(z)=0$. The roots of the quadratic equation $v(x_0+tz)=a$
correspond to the intersection points of the line $L$ and the
ellipsoid $Q_a$. Taking into account that
$v(x_0+tz)=v(x_0)+2\theta_{x_0}(z)t+v(z)t^2$, we obtain the
equivalent equation $2\theta_{x_0}(z)t+v(z)t^2=0$. Since
$q_{a,r}=\{x_0\}$, this quadratic equation has a double root $t=0$,
that is, $\theta_{x_0}(z)=0$. Thus, we obtain $L\subset\theta_{x_0}$
and therefore $\theta_{x_0}=h_r$.

\end{proof}

\begin{corollary} \label{2.2.5} Under conditions {\rm (i)} --
{\rm(iv)} one has $\theta_{x_0}(x)=\frac{a}{r}\pi(x)$.

\end{corollary}

\begin{remark}\label{2.2.10} {\rm If $x_0=0$, then parts {\rm (i)},
{\rm (ii)}, and {\rm (iv)} of Lemma~\ref{2.2.1} hold for $a=r=0$.}

\end{remark}

Let us set $c_p=({}^t\!pQ^{-1}p)^{-1}$,
$E_p^{\left(Q\right)}=\{(a,r)\mid a=c_pr^2,r\geq 0\}$,
$x(a,r)=\frac{a}{r}Q^{-1}p$ for any $(a,r)\in E_p^{\left(Q\right)}$
with $r>0$, $x(0,0)=0$, and
\[
\ef_p^{\left(Q\right)}=\{x\in \hbox{\ccc R}^n\mid x=x(a,r),\hbox{\
}(a,r)\in E_p^{\left(Q\right)}\}.
\]
Thus, the set $\ef_p^{\left(Q\right)}$ consists of all vectors $x\in
\hbox{\ccc R}^n$ which satisfy the four equivalent conditions from
Lemma~\ref{2.2.1}. Note that $0\in\ef_p^{\left(Q\right)}$ and if
$x(a,r)\in\ef_p^{\left(Q\right)}$, then $\{x(a,r)\}=q_{a,r}$. In
other words, Lemma~\ref{2.2.1} implies

\begin{corollary}\label{2.2.15} One has
\[
\ef_p^{\left(Q\right)}=\cup_{r\geq 0,a=c_pr^2}q_{a,r}.
\]

\end{corollary}

In case $M$ is a singleton, Theorem~\ref{1.5.20} yields the
following two corollaries:

\begin{corollary}\label{2.2.20} Let
$x,x_0\in \ef_p^{\left(Q\right)}$, $x=x(a,r)$, $x_0=x(a_0,r_0)$.

{\rm (i)} If $a=a_0$, then $q_{a,r_0}=\{x_0\}$.

{\rm (ii)} If $a>a_0$, then $q_{a,r_0}$ is an ellipsoid in the
hyperplane $h_{r_0}$.

{\rm (iii)} If $a<a_0$, then $q_{a,r_0}=\emptyset$.

\end{corollary}

\begin{corollary}\label{2.2.25} Let
$x,x_0\in \ef_p^{\left(Q\right)}$, $x=x(a,r)$, $x_0=x(a_0,r_0)$.

{\rm (i)} If $r=r_0$, then $q_{a_0,r}=\{x_0\}$.

{\rm (ii)} If $r<r_0$, then $q_{a_0,r}$ is an ellipsoid in the
hyperplane $h_{r_0}$.

{\rm (iii)} If $r>r_0$, then $q_{a_0,r}=\emptyset$.

\end{corollary}

Corollaries~\ref{2.2.20} and~\ref{2.2.25} imply the following two
equivalent propositions:

\begin{proposition}\label{2.2.30} Let
$x,x_0\in \ef_p^{\left(Q\right)}$, $x=x(a,r)$, $x_0=x(a_0,r_0)$. One
has
\[
r_0=\max_{q_{a_0,r}\neq\emptyset}r\hbox{\rm\ and\ }
a_0=\mathop{\min}_{q_{a,r_0}\neq\emptyset}a.
\]

\end{proposition}

\begin{proposition}\label{2.2.35} Given
$x_0\in\ef_p^{\left(Q\right)}$, one has
\[
\pi(x_0)=\max_{x\in\ef_p^{\left(Q\right)}, v\left(x\right)\leq
v\left(x_0\right)}\pi(x) \hbox{\rm\ and\ }
v(x_0)=\mathop{\min}_{x\in\ef_p^{\left(Q\right)},
\pi\left(x\right)\geq \pi\left(x_0\right)}v(x).
\]

\end{proposition}

It turns out that we can trow out the constraint condition
$x\in\ef_p^{\left(Q\right)}$ from Proposition~\ref{2.2.35}. We have
the following theorem (compare, for example, with~\cite[Section
2]{[30]}).

\begin{theorem}\label{2.2.40} Let $x_0\in q_{a_0,r_0}$ and
$r_0\geq 0$. The following six statements are equivalent:

{\rm (i)} One has $x_0\in\ef_p^{\left(Q\right)}$.

{\rm (ii)} One has
\[
\pi(x_0)=\max_{v\left(x\right)\leq a_0}\pi(x) \hbox{\rm\ and\ }
v(x_0)=\mathop{\min}_{\pi\left(x\right)\geq r_0}v(x).
\]

{\rm (iii)} One has
\begin{equation}
\pi(x_0)=\max_{v\left(x\right)\leq a_0}\pi(x).\label{2.2.45}
\end{equation}

{\rm (iv)} One has
\[
\pi(x_0)=\max_{v\left(x\right)=a_0}\pi(x).
\]

{\rm (v)} One has
\[
v(x_0)=\mathop{\min}_{\pi\left(x\right)\geq r_0}v(x).
\]

{\rm (vi)} One has
\[
v(x_0)=\mathop{\min}_{\pi\left(x\right)=r_0}v(x).
\]

\end{theorem}

\begin{proof} Below we prove only these implications which are not
straightforward.

If $r_0=0$ and $x_0=x(a_0,0)\in \ef_p^{\left(Q\right)}$, then
$a_0=0$, $x_0=0$, and the equivalences hold. Now, let $r_0>0$. In
particular, we have $x_0\neq 0$.

${\rm (i)}\Longrightarrow {\rm (ii)}$ According to
Lemma~\ref{2.2.1}, {\rm (iii)}, and Corollary~\ref{2.2.5} we have
$x_0\in q_{a_0,r_0}$ and $\theta_{x_0}(x)=\frac{a_0}{r_0}\pi(x)$.
Let us suppose $v\left(x\right)\leq a_0$ for $x\in \hbox{\ccc R}^n$.
Then Lemma~\ref{2.1.20}, {\rm (i)}, imply $\pi(x)\leq r_0$. Now, let
$\pi\left(x\right)\geq r_0$, that is, $\theta_{x_0}(x)\geq a_0$ for
some $x\in \hbox{\ccc R}^n$. In this case Lemma~\ref{2.1.20}, {\rm
(iii)}, yields $v(x)\geq a_0$.

${\rm (iii)}\Longrightarrow {\rm (i)}$ Let $x_0$ satisfies
condition~(\ref{2.2.45}). Lemma~\ref{2.1.30}, {\rm (i)}, imply
$\theta_{x_0}=h_{r_0}$. Now Lemma~\ref{2.2.1}, {\rm (iii)}, finishes
the proof.

${\rm (v)}\Longrightarrow {\rm (i)}$. Since $Q_{<a_0}\subset
h_{r_0}(<)$, Lemma~\ref{2.1.30} yields $\theta_{x_0}=h_{r_0}$. In
accord with Lemma~\ref{2.2.1}, {\rm (iii)}, part {\rm (i)} holds.

\end{proof}

\section{Markowitz Geometry}

\label{3}

In this section we unite the results from the previous two sections
and give complete characterization of the tangent points of a family
of concentric ellipsoids and a family of parallel hyperplanes in an
affine subspace of $\hbox{\ccc R}^n$.

\subsection{The Equality}

\label{3.1}

Let $M\neq\emptyset$, $\ell\in M$, and let us set $L=\{\ell\}$,
$K=M\setminus L$. Let us fix all components of
$\tau^{\left(K\right)}\in\hbox{\ccc R}^K$:
$\tau^{\left(K\right)}=\mu^{\left(K\right)}$, and set
$h^{\left(K\right)}=h(\mu^{\left(K\right)})$,
$\varrho^{\left(K\right)}=c^{\left(Q;K^c\right)}(\mu^{\left(K\right)})$,
$\gamma^{\left(K\right)}=\gamma_K(\mu^{\left(K\right)})$. We denote
$r=\tau_\ell$, $\varrho=\varrho_\ell^{\left(K\right)}$,
$r'=r-\varrho$, so
$\tau^{\left(M\right)}=\tau^{\left(L,K\right)}(r)$. Finally, we set
$a=\gamma_M(\tau^{\left(L,K\right)}(r))$.

We remind that after the translation $z=x-\varrho^{\left(K\right)}$
of the coordinate system, $(z_s)_{s\in K^c}$, where
$z_s=z_s^{\left(K^c\right)}$, is a system of coordinate functions on
the affine subspace $h^{\left(K\right)}$ with origin
$\varrho^{\left(K\right)}$. In this case $h(\tau^{\left(M\right)})=
h(\tau^{\left(L,K\right)}(r))$ is a hyperplane in
$h^{\left(K\right)}$ with equation $z_\ell=r'$. In particular, the
corresponding $\ell$-th coordinate vector $p\in\hbox{\ccc R}^{K^c}$
(the $\ell$-th component of $p$ is $1$ and all other components are
zeroes) is a normal vector of $h(\tau^{\left(L,K\right)}(r))$ in
$h^{\left(K\right)}$. We set $\pi(x)=x_\ell$,
$\pi^{\left(K^c\right)}(z)=z_\ell$, and note that the linear form
$\pi^{\left(K^c\right)}(z)$ is the restriction on
$h^{\left(K\right)}$ of the linear form $\pi(x)$, written in terms
of $z$. It follows from Corollary~\ref{1.5.55}, {\rm (i)}, that the
trace $q_{a,\mu^{\left(K\right)}}$ of the ellipsoid $Q_a$ on affine
space $h^{\left(K\right)}$ is nonempty and the hyperplane
$h(\tau^{\left(L,K\right)}(r))$ is tangential to the ellipsoid
$q_{a,\mu^{\left(K\right)}}$ at the point
$c^{\left(Q;M^c\right)}(\tau^{\left(L,K\right)}(r))$.

In order to stick together notation from sections~\ref{1}
and~\ref{2} in this case, we set $a'=a-\gamma^{\left(K\right)}$,
$h(\tau^{\left(L,K\right)}(r))=h_{r'}$,
$q_{a',r'}=q_{a,\mu^{\left(K\right)}}\cap
h_{r'}=Q_{a'}^{\left(K^c\right)}\cap h_{r'}$.

\begin{theorem}\label{3.1.1} {\rm (i)} If $r'\geq 0$, then
\begin{equation}
x(a',r')=c^{\left(Q;M^c\right)}(\tau^{\left(L,K\right)}(r))
\label{3.1.5}
\end{equation}
and $x(0,0)=\varrho^{\left(K\right)}$.

{\rm (ii)} One has
\[
\eff_{\ell^{\varrho+}}^{\left(Q;M^c\right)}=
\ef_p^{\left(Q^{\left(K^c\right)}\right)}.
\]

\end{theorem}

\begin{proof} {\rm (i)} The affine space
$h(\tau^{\left(L,K\right)}(r))$ is a hyperplane in
$h^{\left(K\right)}$, which is tangential to the ellipsoid
$q_{a,\mu^{\left(K\right)}}$ at the point
$c^{\left(Q;M^c\right)}(\tau^{\left(L,K\right)}(r))$. In particular,
$Q_{a'}^{\left(K^c\right)}\cap h_{r'}=
\{c^{\left(Q;M^c\right)}(\tau^{\left(L,K\right)}(r))\}$ and
Lemma~\ref{2.2.1}, {\rm (ii)}, yields $a'=c_p{r'}^2$ for
$c_p=({}^t\!pQ^{-1}p)^{-1}$. Therefore, when $r'\geq 0$, we have
$(a',r')\in E_p^{\left(Q\right)}$ and the equality~(\ref{3.1.5})
holds. In addition, if $r'=0$, then $a'=0$,
$\gamma_M(\tau^{\left(L,K\right)}(r))=\gamma^{\left(K\right)}$, and
Corollary~\ref{1.5.55}, {\rm (ii)}, {\rm (iii)}, implies
$\gamma_L^{\left(Q^{\left(K^c\right)}\right)}(\tau^{\left(L\right)}-
c_L^{\left(Q;K^c\right)}(\mu^{\left(K\right)}))=0$, hence
\[
c^{\left(Q^{\left(K^c\right)};M^c\right)}((\tau^{\left(L\right)}-
c_L^{\left(Q;K^c\right)}(\mu^{\left(K\right)})))=0.
\]
In other words,
\[
c^{\left(Q;M^c\right)}(\tau^{\left(L,K\right)}(\varrho))=
c^{\left(Q;K^c\right)}(\mu^{\left(K\right)}).
\]
This shows that $x(0,0)=c^{\left(Q;K^c\right)}
(\mu^{\left(K\right)})=\varrho^{\left(K\right)}$ and the
equality~(\ref{3.1.5}) proves part {\rm (i)} which, in turn, yields
part {\rm (ii)}.

\end{proof}

\begin{theorem}\label{3.1.10}  Let $x_0=
c^{\left(Q;M^c\right)}(\tau^{\left(L,K\right)}(r_0))\in
\eff_{\ell^{\varrho+}}^{\left(Q;M^c\right)}$. One has $r_0=\pi(x_0)$
and if $a_0=v(x_0)$, then
\begin{equation}
\pi(x_0)=\max_{x\in h^{\left(K\right)}, v\left(x\right)\leq
a_0}\pi(x) \hbox{\rm\ and\ } v(x_0)=\mathop{\min}_{x\in
h^{\left(K\right)}, \pi\left(x\right)\geq r_0}v(x).\label{3.1.15}
\end{equation}

\end{theorem}

\begin{proof} According to Theorem~\ref{3.1.1}, we have
$r'_0=r_0-\varrho\geq 0$, hence
$x_0=x(a_0,r_0)\in\ef_p^{\left(Q^{\left(K^c\right)}\right)}$. Let
$x_0=z_0+\varrho^{\left(K\right)}$. Theorem~\ref{2.2.40}, {\rm (i)},
{\rm (ii)}, implies
\[
\pi^{\left(K^c\right)}(z_0)= \max_{z\in h^{\left(K\right)},
v^{\left(K^c\right)}\left(z\right)\leq
a'_0}\pi^{\left(K^c\right)}(z)
\]
and
\[
v^{\left(K^c\right)}(z_0)= \mathop{\min}_{z\in h^{\left(K\right)},
\pi^{\left(K^c\right)}\left(z\right)\geq
r'_0}v^{\left(K^c\right)}(z).
\]
Since $\pi^{\left(K^c\right)}(x)=\pi(z)+\varrho$,
$v(x)=v^{\left(K^c\right)}(z)+\gamma^{\left(K\right)}$ on
$h^{\left(K\right)}$, and since $r'_0=r_0-\varrho$,
$a'_0=a_0-\gamma^{\left(K\right)}$, we establish the extremal
property~(\ref{3.1.15}).

\end{proof}

\subsection{The Interpretation}

\label{3.5}

Let $k$, $m$, and $n$ be integers with $n\geq 2$, $0\leq k<n-1$,
$m=k+1$, and let $M=\{n-k,n-k+1,\ldots,n\}$, $K=\{n-k+1,\ldots,n\}$,
$L=\{n-k\}$. Let $h^{\left(j\right)}\colon
\pi^{\left(j\right)}(y)=\tau_j,\hbox{\ }j\in M$, be linearly
independent affine hyperplanes in $\hbox{\ccc R}^n$. We fix
$h^{\left(n\right)}\colon y_1+\cdots+y_n=1$, so $\tau_n=1$, and
denote this hyperplane by $\Pi$. Since $\pi^{\left(j\right)}(y)$ are
linearly independent linear forms, we can change the coordinates in
$\hbox{\ccc R}^n$: $y=Ax$, in such a way that the hyperplane
$h^{\left(j\right)}$ has equation $x_j=\tau_j$, $j\in M$, and,
moreover, $x_i=y_i$, $i\in [n]\setminus M$.

We fix $\tau^{\left(K\right)}$:
$\tau^{\left(K\right)}=\mu^{\left(K\right)}$ ($\mu_n=1$), and
interpret $h^{\left(n\right)}=\Pi$ as the hyperplane consisting of
all {\it financial portfolios} with $n$ assets (here $y_s$ is the
relative amount of money invested in the $s$-th asset,
$s=1,\ldots,n$). The affine subspace
$h^{\left(n-k+1\right)}\cap\ldots\cap h^{\left(n-1\right)}$ (which
is equal to $ \hbox{\ccc R}^n$ if $m=2$) represents several
additional {\it linear constrain conditions} and its trace on $\Pi$
is the affine space $C=h^{\left(K\right)}=h^{\left(n-k+1\right)}\cap
h^{\left(n-1\right)}\cap\ldots\cap\Pi $ of {\it linear constrain
conditions on $\Pi$}.

We denote $\ell=n-k$, $\pi^{\left(\ell\right)}(y)=\pi(y)$ and let
$r=\tau_{\ell}$ be variable. When the coefficient in front of $y_s$
in the linear form $\pi(y)$ is the expected return on $s$-th asset,
$s=1,\ldots,n$, the trace of the hyperplane
$h=h^{\left(\ell\right)}$, $h\colon\pi(y)=r$, on $\Pi$ may be
interpreted as the set of all financial portfolios with {\it
expected return} $r$. Moreover, the trace of the hyperplane $h$ on
$C$ may be interpreted as the set of all financial portfolios with
{\it expected return} $r$, that obey the above linear constrain
conditions on $\Pi$.

On the other hand, if $v(x)={}^t\!xQx$, where ${}^t\!A^{-1}QA^{-1}$
is the $n\times n$ covariance matrix produced by the expected
returns of the individual assets, we may interpret $v(x)$ as the
risk of the portfolio $x$. Theorem~\ref{3.1.10} yields that the ray
$E=\eff_{\ell^{\varrho+}}^{\left(Q;M^c\right)}$ with endpoint
$\varrho^{\left(K\right)}$ is the locus of all Markowitz efficient
portfolios which satisfy the linear constraint conditions $C$. It
turns out that the value $v(\varrho^{\left(K\right)})$ is the
absolute minimum of the risk and in terms of $x$-coordinates the
$\ell$-th component of $\varrho^{\left(K\right)}$ is the absolute
minimum of the corresponding expected return $r$ under the given
constrains.

In order to relate this approach to the classical one, we have to
study the intersection $E\cap\Delta$, where $\Delta$ is the trace of
the unit simplex in $\Pi$ on $C$, because the members of
$E\cap\Delta$ are the efficient portfolios that have no short sales.
Moreover, the properties of this intersection characterize the
financial market.

\appendix

\section{Appendix}

\label{A}

In this appendix we use freely notation introduced in the main body
of the paper.

\subsection{Three Lemmas}

\label{A.1}

The partition $M^c\cup M=[n]$ of the set of indices $[n]$ produces
the following partitioned matrices: Any vector
$x={}^t(x_1,\ldots,x_n)\in \hbox{\ccc R}^n$ can be visualized as
$x={}^t(x^{\left(M^c\right)},x^{\left(M\right)})$ and any $n\times
n$-matrix $Q$ can be visualized as
\[
\left(
\begin{array}{cccccccccccccccccc}
Q^{\left(M^c\right)} & Q^{\left(M^c\times M\right)} \\
Q^{\left(M\times M^c\right)} & Q^{\left(M\right)} \\
 \end{array}
\right).
\]

\begin{lemma} \label{A.1.1} Let $Q$ be a symmetric $n\times n$-matrix
and let $v(x)={}^txQx$ be the corresponding quadratic form. One has
\[
v(x)={}^tx^{\left(M^c\right)}Q^{\left(M^c\right)}x^{\left(M^c\right)}+
2{}^tx^{\left(M^c\right)}Q^{\left(M^c\times
M\right)}x^{\left(M\right)}+
{}^tx^{\left(M\right)}Q^{\left(M\right)}x^{\left(M\right)}.
\]

\end{lemma}

\begin{proof} We have
\[
v(x)={}^txQx=({}^tx^{\left(M^c\right)},{}^tx^{\left(M\right)})\left(
\begin{array}{cccccccccccccccccc}
Q^{\left(M^c\right)} & Q^{\left(M^c\times M\right)} \\
{}^tQ^{\left(M^c\times M\right)} & Q^{\left(M\right)} \\
 \end{array}
\right){}^t(x^{\left(M^c\right)},x^{\left(M\right)})=
\]
\[
{}^tx^{\left(M^c\right)}Q^{\left(M^c\right)}x^{\left(M^c\right)}+
2{}^tx^{\left(M^c\right)}Q^{\left(M^c\times
M\right)}x^{\left(M\right)}+
{}^tx^{\left(M\right)}Q^{\left(M\right)}x^{\left(M\right)}.
\]

\end{proof}

Below we assume that $Q^{\left(M^c\right)}$ is an invertible matrix.

\begin{lemma}\label{A.1.5} Let
\[
c_{M^c}^{\left(M^c\right)}(x^{\left(M\right)})=
-(Q^{\left(M^c\right)})^{-1}Q^{\left(M^c\times
M\right)}x^{\left(M\right)}, \hbox{\ }
c^{\left(M^c\right)}(x^{\left(M\right)})=
{}^t(c_{M^c}^{\left(M^c\right)}(x^{\left(M\right)}),x^{\left(M\right)}),
\]
\[
\hbox{\rm and let\ }\gamma_M(x^{\left(M\right)})=
-{}^tc_{M^c}^{\left(M^c\right)}(x^{\left(M\right)})Q^{\left(M^c\right)}
c_{M^c}^{\left(M^c\right)}(x^{\left(M\right)})+{}^tx^{\left(M\right)}
Q^{\left(M\right)}x^{\left(M\right)}.
\]
{\rm (i)} $\gamma_M(x^{\left(M\right)})$ is a quadratic form in
$x^{\left(M\right)}$,
\[
\gamma_M(x^{\left(M\right)})=
{}^tx^{\left(M\right)}[Q^{\left(M\right)}-{}^tQ^{\left(M^c\times
M\right)}(Q^{\left(M^c\right)})^{-1}Q^{\left(M^c\times M\right)}
]x^{\left(M\right)},
\]
and one has $\gamma_M(x^{\left(M\right)})=
v(c^{\left(M^c\right)}(x^{\left(M\right)}))$.

{\rm (ii)} If $v(x)$ is a positive definite quadratic form in $x$,
then $\gamma_M(x^{\left(M\right)})$ is a positive definite quadratic
form in $x^{\left(M\right)}$.

\end{lemma}

\begin{proof} {\rm (i)} We begin by noting that since
\[
{}^tc_{M^c}^{\left(M^c\right)}(x^{\left(M\right)})Q^{\left(M^c\right)}
c_{M^c}^{\left(M^c\right)}(x^{\left(M\right)})=
{}^tx^{\left(M\right)}{}^tQ^{\left(M^c\times
M\right)}(Q^{\left(M^c\right)})^{-1} Q^{M^c\times
M}x^{\left(M\right)},
\]
we obtain the above expression for $\gamma_M(x^{\left(M\right)})$.
On the other hand, Lemma~\ref{A.1.1} implies
\[
v(c^{\left(M^c\right)}(x^{\left(M\right)}))=
\]
\[
{}^tc_{M^c}^{\left(M^c\right)}(x^{\left(M\right)})
Q^{\left(M^c\right)}c_{M^c}^{\left(M^c\right)}(x^{\left(M\right)})+
2{}^tc_{M^c}^{\left(M^c\right)}(x^{\left(M\right)})Q^{\left(M^c\times
M\right)}x^{\left(M\right)}+
{}^tx^{\left(M\right)}Q^{\left(M\right)}x^{\left(M\right)}.
\]
Taking into account that $Q^{\left(M^c\times
M\right)}x^{\left(M\right)}=
-Q^{\left(M^c\right)}c_{M^c}^{\left(M^c\right)}
(x^{\left(M\right)})$, we establish the identity.

{\rm (ii)} In is enough to note that
$c^{\left(M^c\right)}(x^{\left(M\right)})=0$ if and only if
$x^{\left(M\right)}=0$.

\end{proof}

Now, let us translate the system of coordinates by the rule
\[
z(\tau^{\left(M\right)})=x-
c^{\left(M^c\right)}(\tau^{\left(M\right)}).
\]

\begin{lemma}\label{A.1.10} If $x^{\left(M\right)}=
\tau^{\left(M\right)}$, then
\[
v(x)={}^tz^{\left(M^c\right)}Q^{\left(M^c\right)}
z^{\left(M^c\right)}+ \gamma_M(\tau^{\left(M\right)}).
\]

\end{lemma}

\begin{proof} In accord with Lemma~\ref{A.1.1}, we have
\[
v(x)={}^tx^{\left(M^c\right)}Q^{\left(M^c\right)}x^{\left(M^c\right)}+
2{}^tx^{\left(M^c\right)}Q^{\left(M^c\times
M\right)}\tau^{\left(M\right)}+
{}^t\tau^{\left(M\right)}Q^{\left(M\right)}\tau^{\left(M\right)}=
\]
\[
{}^tx^{\left(M^c\right)}Q^{\left(M^c\right)}x^{\left(M^c\right)}
-2{}^tx^{\left(M^c\right)}Q^{\left(M^c\right)}
c_{M^c}^{\left(M^c\right)}(\tau^{\left(M\right)})+
{}^t\tau^{\left(M\right)}Q^{\left(M\right)}\tau^{\left(M\right)}=
\]
\[
{}^t(z^{\left(M^c\right)}+
c_{M^c}^{\left(M^c\right)}(\tau^{\left(M\right)}))
Q^{\left(M^c\right)}(z^{\left(M^c\right)}+
c_{M^c}^{\left(M^c\right)}(\tau^{\left(M\right)}))
\]
\[
-2{}^t(z^{\left(M^c\right)}+
c_{M^c}^{\left(M^c\right)}(\tau^{\left(M\right)}))Q^{\left(M^c\right)}
c_{M^c}^{\left(M^c\right)}(\tau^{\left(M\right)})+
{}^t\tau^{\left(M\right)}Q^{\left(M\right)}\tau^{\left(M\right)}=
\]
\[
{}^tz^{\left(M^c\right)}Q^{\left(M^c\right)}z^{\left(M^c\right)}+
{}^tc_{M^c}^{\left(M^c\right)}Q^{\left(M^c\right)}
c_{M^c}^{\left(M^c\right)}+2{}^tz^{\left(M^c\right)}Q^{\left(M^c\right)}
c_{M^c}^{\left(M^c\right)}-
\]
\[
2{}^tz^{\left(M^c\right)}Q^{\left(M^c\right)}
c_{M^c}^{\left(M^c\right)}-
2{}^tc_{M^c}^{\left(M^c\right)}Q^{\left(M^c\right)}
c_{M^c}^{\left(M^c\right)}+
{}^t\tau^{\left(M\right)}Q^{\left(M\right)}\tau^{\left(M\right)}=
\]
\[
{}^tz^{\left(M^c\right)}Q^{\left(M^c\right)}z^{\left(M^c\right)}+
\gamma_M(\tau^{\left(M\right)}).
\]

\end{proof}

\section*{Asknowledgements}

\label{4}

I would like to thank the governing body of the Institute of
Mathematics and Informatics at the Bulgarian Academy of Sciences for
creating perfect conditions of work.

\section*{Funding}

\label{5}

This paper was partially supported by the Bulgarian Science Fund
[grand number I02/18].

\end{document}